\documentclass[reqno,12pt]{amsart}

\usepackage[utf8]{inputenc}
\usepackage{microtype}
\usepackage{amsfonts}
\usepackage{amsmath}
\usepackage{graphicx}
\usepackage{hyperref}
\usepackage{color}
\usepackage{comment}

\usepackage{amssymb}
\usepackage{enumitem}

\usepackage{tikz}
\usetikzlibrary{topaths}

\usepackage{a4wide}

\usepackage{empheq}

\newtheorem{theorem}{Theorem}[section]
\newtheorem{lemma}[theorem]{Lemma}
\newtheorem{proposition}[theorem]{Proposition}
\newtheorem{corollary}[theorem]{Corollary}
\newtheorem{cit}[theorem]{Citation}
\newtheorem{observation}[theorem]{Observation}

\theoremstyle{definition}
\newtheorem{definition}[theorem]{Definition}
\newtheorem{remark}[theorem]{Remark}
\newtheorem{example}[theorem]{Example}

\newtheorem*{main:contractible}{Theorem~\ref{thrm:contractible}}
\newtheorem*{main:asymp_cat0_group}{Theorem~\ref{thrm:asymp_cat0_group}}
\newtheorem*{main:F_infty}{Theorem~\ref{thrm:F_infty}}

\newcommand{\Z}{\mathbb{Z}}
\newcommand{\N}{\mathbb{N}}
\newcommand{\R}{\mathbb{R}}

\newcommand{\lk}{\operatorname{lk}}
\newcommand{\st}{\operatorname{st}}
\newcommand{\dst}{\st^\downarrow\!}
\newcommand{\dlk}{\lk^\downarrow\!}
\newcommand{\dflk}{\lk^\downarrow_\partial}
\newcommand{\dclk}{\lk^\downarrow_\delta}

\newcommand{\rips}{\mathcal{VR}}

\newcommand{\defeq}{\mathbin{\vcentcolon =}}

\DeclareMathOperator{\F}{F}
\DeclareMathOperator{\CAT}{CAT}
\DeclareMathOperator{\diam}{\texttt{diam}}
\renewcommand{\dim}{\texttt{dim}}

\numberwithin{equation}{section}

\begin{document}

\title{Bestvina--Brady discrete Morse theory and Vietoris--Rips complexes}
\date{\today}


\author{Matthew C.~B.~Zaremsky}
\address{Department of Mathematics and Statistics, University at Albany (SUNY), Albany, NY 12222}
\email{mzaremsky@albany.edu}

\begin{abstract}
We inspect Vietoris--Rips complexes $\rips_t(X)$ of certain metric spaces $X$ using a new generalization of Bestvina--Brady discrete Morse theory. Our main result is a pair of metric criteria on $X$, called the \emph{Morse Criterion} and \emph{Link Criterion}, that allow us to deduce information about the homotopy types of certain $\rips_t(X)$. One application is to topological data analysis, specifically persistence of homotopy type for certain Vietoris--Rips complexes. For example we recover some results of Adamaszek--Adams and Hausmann regarding homotopy types of $\rips_t(S^n)$. Another application is to geometric group theory; we prove that any group acting geometrically on a metric space satisfying a version of the Link Criterion admits a geometric action on a contractible simplicial complex, which has implications for the finiteness properties of the group. This applies for example to asymptotically $\CAT(0)$ groups. We also prove that any group with a word metric satisfying the Link Criterion in an appropriate range has a contractible Vietoris--Rips complex, and use combings to exhibit a family of groups with this property.
\end{abstract}

\maketitle
\thispagestyle{empty}

\section*{Introduction}

The Vietoris--Rips complex $\rips_t(X)$ of a metric space $X$ is the simplicial complex whose vertex set is $X$ and whose simplices are given by collections of elements of $X$ that are pairwise not more than distance $t$ apart. This is a natural way to extract a nice simplicial complex from a possibly strange metric space. Varying $t$ gives us a filtration $\{\rips_t(X)\}_{t\in\R}$ of the ``simplex on $X$'' $\rips_\infty(X)$, which is contractible (we assume all metric spaces here are non-empty). The topological properties of the $\rips_t(X)$ are generally quite difficult to analyze. For example, for the $n$-sphere $S^n$ a complete picture of the homotopy types of $\rips_t(S^n)$ for all $t$ is currently known only in the $n=1$ case \cite{adamaszek17circle}.

For bounded $X$ the $\rips_t(X)$ are contractible for large enough $t$, but for unbounded $X$ this need not be true. One can ask then what sorts of conditions on $X$ ensure that $\rips_t(X)$ is contractible for large enough $t$, or more generally what sorts of conditions allow us to compute homotopy types of any of the $\rips_t(X)$. In this paper we approach this problem using discrete Morse theory, or more precisely a new generalization of Bestvina--Brady discrete Morse theory. Discrete Morse theory is a powerful tool that leverages ``local'' topological information to make ``global'' conclusions about, broadly speaking, subcomplexes of CW complexes. Our new definition of ``discrete Morse function'' in Definition~\ref{def:morse} simultaneously generalizes a number of previous notions of discrete Morse function, e.g, those in \cite{brown92,bestvina97,forman98,zaremsky17sepPB,witzel18}.

Our main results hinge on a pair of criteria that a metric space $X$ could satisfy, which we call the \emph{Morse Criterion} and \emph{Link Criterion}. Intuitively, the Morse Criterion asks that $X$ be somewhat uniformly discrete, and the Link Criterion asks that intersections of metric balls have controllable diameter. A standard example of $X$ satisfying both criteria is $\Z^n$ with the euclidean metric induced from $\R^n$ (see Example~\ref{ex:lattice}). The main result of the paper is:

\begin{main:contractible}
Let $X$ be a metric space satisfying the Morse Criterion, and the Link Criterion in the range $I$. Then for any $(t,s]\subseteq I$, the inclusion $\rips_t(X)\to\rips_s(X)$ is a homotopy equivalence, and for any $(t,\infty)\subseteq I$, $\rips_t(X)$ is contractible.
\end{main:contractible}

The discrete Morse function $(\diam,-\dim)$ that leads to Theorem~\ref{thrm:contractible} also has applications to topological data analysis. An example of an application in this world is a shorter, ``local'' proof of Hausmann's theorem on Vietoris--Rips complexes of spheres \cite{hausmann95}. Namely, we prove that if $S^n$ has the arclength metric, scaled so antipodal points have distance $1/2$, then for any $0<t<1/4$ we have $\rips_t(S^n)\simeq S^n$ (Proposition~\ref{prop:spheres}). For $n=1$ we get the improved range $0<t<1/3$, recovering part of a result of Adamaszek--Adams \cite{adamaszek17circle}. For $n=1$ the homotopy type is actually known for all $t$ \cite{adamaszek17circle}, and it would be interesting to try and use this Morse theoretic approach to finish the picture for all $n$ and all $t$.

We also discuss applications in the world of geometric group theory, and topological finiteness properties of groups. Recall that a group is of \emph{type $F_n$} if it admits a geometric (that is, proper and cocompact) action on an $(n-1)$-connected CW complex (a more standard definition of type $\F_n$ asks that the action be free, but in fact properness suffices, for example thanks to Brown's Criterion \cite{brown87}). A group is of \emph{type $F_\infty$} if it is of type $\F_n$ for all $n$. We will also say a group is of \emph{type $F_*$} if it admits a geometric action on a contractible CW complex, which is stronger than type $\F_\infty$, for example Thompson's group $F$ is of type $\F_\infty$ \cite{brown84}, but not $\F_*$ since it contains $\Z^\infty$. (The name ``type $\F_*$'' is not widely used, but it was recently coined by Craig Guilbault, and will be useful here to state some upcoming results.) Finally, we mention that a group is of \emph{type $F$} if it acts freely and cocompactly on a contractible CW complex, which is equivalent to being torsion-free and of type $\F_*$.

One result, Proposition~\ref{prop:geom_action}, says that a group acting geometrically on a proper metric space satisfying the so called Arbitrarily Pinched Strong Link Criterion admits a geometric action on a contractible simplicial complex and hence is of type $\F_*$. This has the following implication for asymptotically $\CAT(0)$ groups (as in \cite{kar11}):

\begin{main:asymp_cat0_group}
Let $G$ be an asymptotically $\CAT(0)$ group. Then $G$ admits a geometric action on a contractible simplicial complex, and hence is of type $\F_*$.
\end{main:asymp_cat0_group}

This generalizes some previously known results, e.g., that hyperbolic and $\CAT(0)$ groups are of type $\F_*$ \cite{ontaneda05,bridson99}, and that asymptotically $\CAT(0)$ groups are of type $\F_\infty$ \cite{kar11}.

We also get the following result involving word metrics:

\begin{main:F_infty}
Let $G$ be a group with a word metric corresponding to some finite generating set. If $G$ satisfies the Link Criterion in the range $[N,\infty)$ for some $N$ then $\rips_t(G)$ is contractible for all $t\ge N$, and hence $G$ is of type $\F_*$.
\end{main:F_infty}

We discuss some examples of groups satisfying Theorem~\ref{thrm:F_infty}, in particular examples involving combings of groups. As a special case, this recovers Rips' result on hyperbolic groups admitting contractible Vietoris--Rips complexes.

\medskip

Since different parts of this paper might be of interest to readers with quite different backgrounds, let us point out what might be interesting to whom. Topological data analysts may be most interested in the reconciliation of Forman's discrete Morse theory for finite simplicial complexes as a special case of Bestvina--Brady discrete Morse theory, in Section~\ref{sec:morse} and Remark~\ref{rmk:forman_to_bb}, and also in the applications, e.g., to persistence of homotopy type, in Section~\ref{sec:tda}. Metric geometers might be interested in the Morse Criterion and Link Criterion in Sections~\ref{sec:rips} and~\ref{sec:links} in and of themselves, and the implications for general Vietoris--Rips complexes. Geometric group theorists will probably be most interested in the examples in Section~\ref{sec:link_examples} involving generalizations of $\CAT(0)$ spaces, and the applications to finiteness properties of groups in Section~\ref{sec:groups}, and perhaps also the new generalization of ``discrete Morse function'' in Definition~\ref{def:morse}.

\subsection*{Acknowledgments} The idea for this project stemmed from a conversation with Henry Adams at the Upstate New York Topology Seminar, and I am grateful to Henry for enthusiastically wondering what Bestvina--Brady Morse theory could do in the world of Vietoris--Rips complexes. I also want to thank him for a number of excellent questions and comments on earlier drafts, including asking a question that led to Corollary~\ref{cor:approx}. I am also grateful to Craig Guilbault for explaining some terminological background, Susan Hermiller for helping me track down a reference, and Matt Brin and Vidit Nanda for helpful suggestions. I also thank the anonymous referees for many helpful comments, and especially for catching a mistake in an earlier version of Observation~\ref{obs:discrete}.

\section{Discrete Morse theory}\label{sec:morse}

In this section we discuss Bestvina--Brady discrete Morse theory (also called PL Morse theory), in a generality applicable to the sorts of problems we consider here. Bestvina--Brady discrete Morse theory was developed by Bestvina and Brady in \cite{bestvina97}, and has proven to be an indispensable tool in the study of topological aspects of discrete groups. A nice introduction to the theory is given in \cite{bestvina08}. We should mention that in the literature the term ``discrete Morse theory'' often means Forman's discrete Morse theory \cite{forman98}; see Remark~\ref{rmk:forman_to_bb} for a discussion of the relationship.

The basic purpose of discrete Morse theory is to understand topological properties of subcomplexes of a complex. Given a CW complex $X$ and a well order on the cells, one can consider all the \emph{sublevel complexes}, that is the subcomplexes consisting of just those cells less than (or less than or equal to) a given cell in the order. Thanks to well ordering, at every stage there is a well defined ``next'' cell to glue in, so transfinite induction lets us make conclusions about the sublevel complexes. If one understands the relative links along which all cells are attached, one can in theory deduce the homotopy type of any sublevel complex.

In practice this is usually impossibly complicated, but discrete Morse theory provides a way to implicitly choose a well order, using a globally defined function called a Morse function, so that every relative link of a cell equals the so called ``descending link'' of that cell with respect to the Morse function. In concrete applications the descending links are often quite nice, and so Morse theory turns the difficult global problem of understanding sublevel complexes into an easier local problem about descending links. The key ideas behind Bestvina--Brady-style discrete Morse theory include that the cells of the complex have an affine structure, and the Morse function is affine and non-constant when restricted to each positive dimensional cell. This ensures that the descending links are determined by their $0$-skeleta and that one only needs to understand descending links of vertices, making the whole problem more tractable. In practice one often hopes for a large number of vertices to have a contractible descending link, since attaching such a vertex does not change the homotopy type.

The original formulation of Bestvina--Brady discrete Morse theory requires a Morse function that is non-constant on edges, affine on cells, and whose image of the vertex set is closed and discrete in $\R$. To handle some recent applications involving Bieri--Neumann--Strebel--Renz invariants of groups, Stefan Witzel and the author developed ways to relax these conditions \cite{witzel18,zaremsky17,zaremsky17sepPB}. For our present purposes we need to relax the conditions even further. The definition of Morse function that will be most useful here is as follows. As we will see after the definition, special cases include both Bestvina--Brady discrete Morse theory and Forman's discrete Morse theory.

\begin{definition}[Morse function]\label{def:morse}
Let $Y$ be an affine cell complex and let $(h,f) \colon Y \to \R\times \R$ be a map such that the restrictions of $h$ and $f$ to any cell are affine functions. Assume for each edge $\{v,w\}$ that $(h,f)(v)\ne (h,f)(w)$. Let $<$ be the lexicographic order on $\R\times\R$. By a \emph{ray} we mean a sequence $v_1,v_2,\dots$ of vertices of $Y$ such that each $\{v_i,v_{i+1}\}$ is an edge in $Y$. Call a ray \emph{descending} if $(h,f)(v_i)>(h,f)(v_{i+1})$ for all $i$, and \emph{ascending} if $(h,f)(v_i)<(h,f)(v_{i+1})$ for all $i$. We call $(h,f)$ a \emph{descending-type Morse function} if for every descending ray $v_1,v_2,\dots$ the $h$ values $h(v_1),h(v_2),\dots$ have no lower bound in $\R$, and alternately an \emph{ascending-type Morse function} if along every ascending ray the $h$ values have no upper bound.
\end{definition}

Note that if there are no descending (respectively ascending) rays to begin with, then this last condition holds vacuously.

If $Y$ is a simplicial complex then any function $Y^{(0)}\to\R$ can be extended to a map $Y\to\R$ by extending affinely to positive dimensional simplices. With this in mind, in the simplicial case we will frequently conflate a Morse function with its restriction to the vertex set, implicitly understanding that the latter is being affinely extended to the positive dimensional simplices.

\begin{example}\label{ex:wz}
Definition~\ref{def:morse} is a generalization of Definition~2.1 in \cite{zaremsky17sepPB} (which was itself a generalization of the notion of Morse function in \cite{witzel18}). There the rules for a descending-type Morse function were that the $f$ values be well ordered in $\R$, and that there exist $\varepsilon>0$ such that for all adjacent vertices $v$ and $w$ either $|h(v)-h(w)|\ge\varepsilon$, or else $h(v)=h(w)$ and $f(v)\ne f(w)$. These conditions clearly imply our current ones.
\end{example}

\begin{example}[Bestvina--Brady]\label{ex:bb}
Bestvina and Brady's original definition of Morse function in \cite{bestvina97} is the special case of Definition~\ref{def:morse} where $h(Y^{(0)})\subseteq \R$ is closed and discrete, $h$ takes distinct values on adjacent vertices, and $f$ is constant. (As a remark, the condition that $h(Y^{(0)})\subseteq \R$ is closed was not stated in \cite{bestvina97}, just that it is discrete, but closedness is necessary for their results to hold.)
\end{example}

A situation that is often relevant (and will be for us) is when $Y$ is the barycentric subdivision of some complex $Z$, so the vertices of $Y$ represent the cells of $Z$. In this case one can use the measurement ``dimension'' as the function $f$ (or ``negative dimension'' depending on whether one wants to be in the descending or ascending scenario) and for $h$ some function defined on the set of simplices of $Z$.

\begin{example}[Dimension]\label{ex:dim}
The most simplistic example of the above is to just use dimension and nothing else. Let $Z'$ be the barycentric subdivision of $Z$, let $h\colon Z'\to\R$ be some constant function, and let $f=\dim\colon Z'\to \R$ send each vertex in $Z'$ to its dimension as a cell of $Z$. Then $(h,\dim)$ is a descending-type Morse function, but it is not very useful, for reasons we will point out later (to foreshadow, none of the descending links are contractible).
\end{example}

\begin{example}[Forman]\label{ex:forman}
Forman's discrete Morse theory for finite simplicial complexes, introduced in \cite{forman98}, is also a special case. Let us use the setup from Forman's ``user's guide'' \cite{forman02}. Let $Z$ be a finite simplicial complex and $K$ its set of simplices (so $K$ is the vertex set of the barycentric subdivision $Z'$). If $\sigma\in K$ is $k$-dimensional we may write $\sigma^{(k)}$. A ``Forman style'' discrete Morse function is a map $h\colon K\to\R$, such that for each $\sigma^{(k)}\in K$ we have
\[
|\{\widetilde{\sigma}^{(k+1)}>\sigma^{(k)}\mid h(\widetilde{\sigma})\le h(\sigma)\}|\le 1 \text{ and } |\{(\sigma')^{(k-1)}<\sigma^{(k)}\mid h(\sigma')\ge h(\sigma)\}|\le 1\text{.}
\]
It turns out no $\sigma$ can have both of these cardinalities equal to $1$, so the simplices are partitioned into \emph{redundant simplices} (those where the first set above is non-empty), \emph{collapsible simplices} (those where the second set above is non-empty), and \emph{critical simplices} (those where both sets are empty). To recast this in our current framework, view $h$ as a function from the vertex set of the barycentric subdivision $Z'$ of $Z$ to $\R$. Now $(h,-\dim)$ is a descending-type Morse function in the sense of Definition~\ref{def:morse} (since there are no descending rays). It also encodes the exact same information as $h$ regarding relative heights of adjacent simplices. (Incidentally, if $h$ is injective as a function $K\to\R$ then there is not even any need to use the $-\dim$ factor.)
\end{example}

\begin{remark}[Forman to Bestvina--Brady]\label{rmk:forman_to_bb}
For $Z$ a finite simplicial complex, the function $(h,-\dim)$ from Example~\ref{ex:forman} can be converted into a discrete Morse function as in Bestvina--Brady's original definition. Indeed, the outputs of $(h,-\dim)$ on vertices of $Z'$ form a finite totally ordered set, which can be embedded in an order-preserving way into $\R$, so one can view $(h,-\dim)$ as a function to $\R$ that satisfies all the requirements to be a Bestvina--Brady-style discrete Morse function on $Z'$. In other words, any time Forman's discrete Morse theory is being applied to a finite simplicial complex, one could equivalently apply Bestvina--Brady discrete Morse theory. On the other hand, $(h,-\dim)$ is always a Bestvina--Brady-style discrete Morse function for any $h$, even if $h$ is not a Forman-style discrete Morse function.
\end{remark}

\begin{example}[Brown/Brown--Geoghegan]\label{ex:brown}
Related to Forman's discrete Morse theory (and actually predating it) is Brown's notion of a ``collapsing scheme'' \cite{brown92}, which was first used by Brown and Geoghegan in \cite{brown84} (phrased very differently) to prove that Thompson's group $F$ is of type $\F_\infty$. The idea is to ``match'' collapsible and redundant simplices in a simplicial complex $Z$, just like in Forman's discrete Morse theory, with the difference being that $Z$ might be infinite. The additional restriction needed to account for the infinite case, called (C2) in \cite{brown92}, is that for any sequence of redundant simplices $\sigma_1,\sigma_2,\dots$, such that for each $i$, $\sigma_{i+1}$ is a face of the collapsible simplex matched with $\sigma_i$, the sequence must be finite. Using $(h,-\dim)$ as in Example~\ref{ex:forman}, $(h,-\dim)$ is a descending-type Morse function in the sense of Definition~\ref{def:morse} if and only if there are no descending rays, and this is easily seen to be the same as Brown's (C2) condition.
\end{example}

Before stating and proving the Morse Lemma for our current definition of Morse function, we need to discuss ascending and descending links. Given a Morse function $(h,f) \colon Y \to \R\times \R$, since $h$ and $f$ are affine on cells and $(h,f)$ is non-constant on edges, $(h,f)$ restricted to a cell $c$ achieves its maximum and minimum values at unique vertices of $c$. Now define the \emph{descending star} $\dst(v)$ of a vertex $v$ to be the subcomplex of the star $\st_Y(v)$ consisting of those cells on which $(h,f)$ achieves its maximum at $v$, and the \emph{descending link} $\dlk(v)$ of $v$ to be the link of $v$ in $\dst(v)$. The \emph{ascending star} and \emph{ascending link} are defined analogously, but in this paper we will always use the descending setup. Given $t\in\R$ we denote by $Y_{h\le t}$ the full subcomplex of $Y$ spanned by vertices $v$ with $h(v)\le t$ (with $Y_{t\le h}$, $Y_{h<t}$, and $Y_{t<h}$ defined analogously).

For $t<s$ write $Y_{t<h\le s}$ for $Y_{t<h}\cap Y_{h\le s}$. Note that the above setup ensures that for $t<s$, $Y_{h\le s}$ equals the union of $Y_{h\le t}$ with the descending stars of all the vertices in $Y_{t<h\le s}$. Hence to understand the relationship between the homotopy types of $Y_{h\le t}$ and $Y_{h\le s}$, we would like to attach these vertices in an order such that each such vertex gets glued in along a relative link equal to its descending link. This will ensure that $Y_{h\le s}$ is obtained from $Y_{h\le t}$ by coning off all these descending links, so if the descending links have nice enough homotopy types we will consequently understand how the homotopy type changes from $Y_{h\le t}$ to $Y_{h\le s}$ (see, e.g., Corollary~\ref{cor:morse}). The following Morse Lemma says such an order always exists.

\begin{lemma}[Morse Lemma]
Let $(h,f) \colon Y \to \R\times \R$ be a descending-type Morse function and let $t<s$. Then there is a well order on the vertices of $Y_{t<h\le s}$ such that upon attaching these vertices to $Y_{h\le t}$ in this order to obtain $Y_{h\le s}$, the relative link of each vertex equals its descending link.
\end{lemma}

\begin{proof}
Since adjacent vertices have different $(h,f)$ values, it suffices to prove that there exists a well order $\preceq$ on the vertices of $Y_{t<h\le s}$ such that whenever two such vertices $v,w\in Y_{t<h\le s}^{(0)}$ are adjacent and satisfy $(h,f)(v)<(h,f)(w)$, we have $v\prec w$. First define a partial order $\sqsubseteq$ on $Y_{t<h\le s}^{(0)}$ by declaring that $v\sqsubseteq w$ if there exists a sequence $v=v_1,\dots,v_k=w$ of vertices in $Y_{t<h\le s}^{(0)}$ such that for each $1\le i\le k-1$, $v_i$ is adjacent to $v_{i+1}$ and we have $(h,f)(v_i)<(h,f)(v_{i+1})$. In particular whenever $v,w\in Y_{t<h\le s}^{(0)}$ are adjacent and satisfy $(h,f)(v)<(h,f)(w)$, we have $v\sqsubset w$. If $v_1\sqsupset v_2\sqsupset \cdots$ is a strictly decreasing chain then $v_1,v_2,\dots$ are contained in a descending ray, so Definition~\ref{def:morse} ensures that the $h$ values of the $v_i$ eventually drop below $t$; in particular any such chain is finite. We can extend the partial order $\sqsubseteq$ on $Y_{t<h\le s}^{(0)}$ to a total order $\preceq$ via the Szpilrajn Theorem while preserving the property that any strictly decreasing chain $v_1\succ v_2 \succ \cdots$ is finite; see for example \cite{bonnet82}. In particular $\preceq$ is a well order, and since it extends $\sqsubseteq$ it is true that whenever $v,w\in Y_{t<h\le s}^{(0)}$ are adjacent and satisfy $(h,f)(v)<(h,f)(w)$, we have $v\prec w$.
\end{proof}

For example if all the descending links are contractible, then the homotopy type does not change from $Y_{h\le t}$ to $Y_{h\le s}$. More generally, standard tools like the Seifert--van Kampen Theorem, Mayer--Vietoris sequence, and Hurewicz Theorem can reveal how the homotopy type changes upon coning off a descending link. In practice, one tends to hope that many of the descending links are contractible.

Let us look at some previous examples and see how descending links behave.

\begin{example}
In Example~\ref{ex:dim} using nothing but $\dim$, the descending link of every simplex in $Z$ (i.e., vertex of the barycentric subdivision $Z'$) is just the boundary of the simplex. As we said earlier, this is not especially useful, since no descending links are contractible.
\end{example}

\begin{example}[Forman descending links]
Consider Example~\ref{ex:forman}, with the Forman-style Morse function $h\colon (Z')^{(0)}\to\R$ and the Definition~\ref{def:morse}-style Morse function $(h,-\dim)\colon Z'\to\R\times\R$. It is easy to see that if a simplex $\sigma$ has a proper face $\sigma'$ with $h(\sigma')\ge h(\sigma)$ then it has a codimension-$1$ such face containing all such faces. Similarly if $\sigma$ has a proper coface $\widetilde{\sigma}$ with $h(\widetilde{\sigma})\le h(\sigma)$ then it has a codimension-$1$ such coface contained in all such cofaces. This shows that the descending link of a redundant simplex $\sigma$ is the join of its boundary with a cone on the one codimension-$1$ coface $\widetilde{\sigma}$ with $h(\widetilde{\sigma})\le h(\sigma)$, hence is contractible. Also, the descending link of a collapsible simplex $\sigma$ is contained in the boundary of $\sigma$ minus the one codimension-$1$ face $\sigma'$ with $h(\sigma')\ge h(\sigma)$, so is a cone on the one vertex of $\sigma$ not contained in $\sigma'$. In particular the descending link is contractible in both of these cases. The descending link of a critical simplex is simply its boundary. All of this syncs up with the situation framed using Forman's language; in particular one can remove all redundant and collapsible simplices along their relative links (which equal their descending links) and get a homotopy equivalent complex corresponding to just the critical simplices. Similar observations hold for the situation from Example~\ref{ex:brown}.
\end{example}

If one cares less about precise homotopy type and more about high degrees of connectedness, then the following is a particularly useful corollary of the Morse Lemma.

\begin{corollary}\label{cor:morse}
If for all vertices $v$ with $t<h(v)\le s$ we have that $\dlk(v)$ is $(n-1)$-connected then the inclusion $Y_{h\le t}\to Y_{h\le s}$ induces an isomorphism in $\pi_k$ for $k\le n-1$ and a surjection in $\pi_n$.
\end{corollary}

\begin{proof}
By the Morse Lemma, we can build up from $Y_{h\le t}$ to $Y_{h\le s}$ by gluing in the missing vertices (together with their descending stars) along their descending links. Hence whenever we attach a new vertex we are coning off an $(n-1)$-connected relative link, which induces an isomorphism in $\pi_k$ for $k\le n-1$ and a surjection in $\pi_n$. By transfinite induction (which applies since the order of attaching all the vertices in $Y_{t<h\le s}$ is a well order) we conclude that the inclusion $Y_{h\le t}\to Y_{h\le s}$ also induces these sorts of maps, and we are done.
\end{proof}

\begin{example}
Using a Morse function given by nothing but $\dim$ on the barycentric subdivision, as in Example~\ref{ex:dim}, Corollary~\ref{cor:morse} implies the standard fact that a CW complex is $(n-1)$-connected if and only if its $n$-skeleton is.
\end{example}

\section{Vietoris--Rips complexes and the Morse condition}\label{sec:rips}

Given a metric space $X$ and a parameter $t$, a natural simplicial flag complex one can produce is the Vietoris--Rips complex $\rips_t(X)$. (A simplicial complex is \emph{flag} if every finite collection of vertices pairwise spanning edges spans a simplex.) We recall the definition here.

\begin{definition}[Vietoris--Rips complex]\label{def:rips}
Let $(X,d)$ be a metric space and $t\in \R\cup\{\infty\}$. The \emph{Vietoris--Rips complex with parameter $t$}, denoted $\rips_t(X)$, is the simplicial flag complex whose vertex set is $X$ and whose edge set consists of all $\{x,x'\}$ such that $d(x,x')\le t$. When $t=\infty$ we will write $\rips(X)$ for $\rips_\infty(X)$. For $t<\infty$ we will refer to $\rips_t(X)$ as a \emph{proper} Vietoris--Rips complex for $X$.
\end{definition}

Note that if $t\le 0$ then $\rips_t(X)$ just equals its vertex set $X$. Also note that for $t\le t'$ we have $\rips_t(X)\le \rips_{t'}(X)$, and that $\rips(X)$ is the infinite simplex on the vertex set $X$. We will view $\{\rips_t(X)\}_{t\in\R}$ as a filtration of the contractible complex $\rips(X)$.

Consider the barycentric subdivision $\rips(X)'$ of $\rips(X)$. The vertices of $\rips(X)'$ are the simplices of $\rips(X)$, and the simplices of $\rips(X)'$ are chains $\sigma_0<\cdots<\sigma_k$ of simplices of $\rips(X)$. Given a vertex $\sigma$ of $\rips(X)'$ let $\diam(\sigma)$ be the diameter of the set of vertices of $\sigma$ as a subset of $X$, and let $\dim(\sigma)$ be the dimension of $\sigma$ as a simplex in $\rips(X)$. (As usual we also denote by $\diam$ and $\dim$ these functions extended affinely to the simplices of $\rips(X)'$.) The main function of interest in all that follows is
\[
(\diam,-\dim)\colon \rips(X)'\to \R\times\R \text{.}
\]
If we view $\R\times\R$ as being lexicographically ordered, then $\rips_t(X)'$ is the sublevel complex of $\rips(X)'$ determined by the rule $(\diam,-\dim)(\sigma)\le (t,0)$. In particular one might now hope to use Morse theory to understand the $\rips_t(X)'$, and consequently the $\rips_t(X)$, since $\rips_t(X)\cong \rips_t(X)'$. The following Morse Criterion will turn out to precisely describe when $(\diam,-\dim)$ is a descending-type Morse function (see Lemma~\ref{lem:morse_criterion}).

\begin{definition}[Morse Criterion]
We say a metric space $X$ satisfies the \emph{Morse Criterion} if every ball is finite, and there does not exist an infinite alternating chain $\sigma_1<\sigma_2>\sigma_3<\cdots$ of simplices in $\rips(X)$ with $\diam(\sigma_i)=\diam(\sigma_{i+1})$ for all odd $i$ and $\diam(\sigma_i)>\diam(\sigma_{i+1})$ for all even $i$.
\end{definition}

This is a bit difficult to handle in practice, so let us also present a straightforward sufficient condition. Recall that a metric space $X$ is \emph{proper} if closed balls are compact. Let us call $X$ \emph{topologically discrete} if the metric generates the discrete topology (``discrete metric space'' usually means all non-zero distances are $1$ so we will not use that term). Note that if the set of diameters of finite subsets of $X$ is discrete in $\R$ then $X$ is topologically discrete, but the converse is not true in general (for instance $X=\{1,1+\frac{1}{2},1+\frac{1}{2}+\frac{1}{3},\dots\}$ with the metric induced from the usual metric on $\R$).

\begin{observation}[Sufficient condition]\label{obs:discrete}
Let $X$ be a proper metric space in which the set of diameters of finite subsets of $X$ is closed and discrete in $\R$. Then $X$ satisfies the Morse Criterion.
\end{observation}

\begin{proof}
First, the assumption ensures that $X$ is topologically discrete, so closed balls are compact and discrete, hence finite. Next, given an infinite chain of simplices as in the Morse Criterion, $\diam(\sigma_1),\diam(\sigma_3),\diam(\sigma_5),\dots$ is a strictly decreasing sequence. Any such sequence in $[0,\infty)$ cannot lie in a closed discrete subset.
\end{proof}

Of course the set of diameters of finite subsets of $X$ coincides with the set of distances between points in $X$, which may seem like a better way to phrase things, but phrasing things in terms of diameters will be convenient here.

An important example to keep in mind of a metric space satisfying the conditions in Observation~\ref{obs:discrete} is the vertex set of a connected locally finite graph, with the metric induced from the path metric on the graph. We remark that metric spaces can satisfy the Morse Criterion without satisfying the condition in Observation~\ref{obs:discrete}, for instance this is the case for $X=\{1,1+\frac{1}{2},1+\frac{1}{2}+\frac{1}{3},\dots\}$. Also, one can check that the finite balls condition and the alternating chain condition in the Morse Criterion do not imply one another.

\begin{remark}\label{rmk:geodesic}
Somewhat opposite to metric spaces with discrete sets of diameters of finite subsets is geodesic metric spaces, where a continuum of diameters is possible. No non-trivial geodesic metric space can possibly satisfy the Morse Criterion. First of all, it (badly) fails the requirement that balls be finite. Secondly, it can admit a ``bad'' alternating chain, for example $\{0,\frac{1}{2},1\}\ge \{0,\frac{1}{2}\}\le \{0,\frac{1}{3},\frac{1}{2}\}\ge\cdots$ in $\rips([0,1])$. However, as Lemma~\ref{lem:lattice} will later explain, we will be very much interested in certain subspaces of geodesic spaces, which can satisfy the Morse Criterion.
\end{remark}

The purpose of the Morse Criterion is the following:

\begin{lemma}\label{lem:morse_criterion}
The function $(\normalfont{\diam,-\dim})\colon \rips(X)'\to\R\times\R$ is a descending-type Morse function if and only if the Morse Criterion holds for $X$.
\end{lemma}

\begin{proof}
First assume the Morse Criterion holds. By construction $\diam$ and $-\dim$ are affine on simplices. Adjacent vertices in $\rips(X)'$ have different $-\dim$ values, hence different $(\diam,-\dim)$ values. We need to show that along any descending ray the $\diam$ values have no lower bound; since $\diam$ is bounded below by $0$ this is actually equivalent to saying that no descending rays exist. Suppose a descending ray $\sigma_1,\sigma_2,\dots$ does exist, so for each $i$ we have $(\diam,-\dim)(\sigma_i)>(\diam,-\dim)(\sigma_{i+1})$, and note that either $\sigma_i<\sigma_{i+1}$ or $\sigma_i>\sigma_{i+1}$. For a given $i$, we either have $\diam(\sigma_i)>\diam(\sigma_{i+1})$ in which case $\sigma_i>\sigma_{i+1}$, or $\diam(\sigma_i)=\diam(\sigma_{i+1})$ and $-\dim(\sigma_i)>-\dim(\sigma_{i+1})$ in which case $\sigma_i<\sigma_{i+1}$. Say that the ray \emph{turns up} at $i$ if $\sigma_{i-1}>\sigma_i<\sigma_{i+1}$, \emph{turns down} at $i$ if $\sigma_{i-1}<\sigma_i>\sigma_{i+1}$, \emph{continues up} at $i$ if $\sigma_{i-1}<\sigma_i<\sigma_{i+1}$, and \emph{continues down} at $i$ if $\sigma_{i-1}>\sigma_i>\sigma_{i+1}$. If there exists $n$ such that the ray continues up for all $i\ge n$ then there is a ball in $X$ with infinitely many elements, which the Morse Criterion rules out. It is impossible that there exists $n$ such that the ray continues down for all $i\ge n$, since no simplicial complex allows for an infinite sequence of proper faces. The remaining case is that for all $n$ there exist $i,j\ge n$ such that the ray turns up at $i$ and turns down at $j$. Then by passing to a subsequence we can assume without loss of generality that $\sigma_1<\sigma_2>\sigma_3<\cdots$. But now $\diam(\sigma_i)=\diam(\sigma_{i+1})$ for all odd $i$ and $\diam(\sigma_i)>\diam(\sigma_{i+1})$ for all even $i$, which contradicts the Morse Criterion.

Now suppose the Morse Criterion fails. First assume there is a ball with infinitely many points in it, so there exist $x,y\in X$ with infinitely many $z_1,z_2,\dots$ satisfying $d(x,z_i)\le d(x,y)$. Then $\{x,y\}<\{x,y,z_1\}<\{x,y,z_1,z_2\}<\cdots$ is a descending ray along which $\diam$ is constant, so $(\diam,-\dim)$ is not a descending-type Morse function. Now assume there is an infinite alternating chain $\sigma_1<\sigma_2>\sigma_3<\cdots$ of simplices in $\rips(X)$ with $\diam(\sigma_i)=\diam(\sigma_{i+1})$ for all odd $i$ and $\diam(\sigma_i)>\diam(\sigma_{i+1})$ for all even $i$. Then $\sigma_1,\sigma_2,\dots$ form a descending ray in $\rips(X)'$ along which $\diam$ is bounded below (by $0$), so $(\diam,-\dim)$ is not a desending-type Morse function.
\end{proof}

While $(\diam,-\dim)$ is certainly not the only possible Morse function realizing the $\rips_t(X)'$ as sublevel complexes, it has especially natural descending links: the ``descending moves'' are that one can either add new vertices without raising the diameter, or remove vertices to strictly lower the diameter. The reader may have noticed that $(\diam,\dim)$ also realizes the $\rips_t(X)'$ as sublevel complexes (and even avoids having to assume the analog of the Morse Criterion to be a Morse function), but the descending link of $\sigma$ with respect to $(\diam,\dim)$ is just the boundary of $\sigma$, and this is not especially useful for anything since then there are no contractible descending links.

\subsection{Descending links}\label{sec:dlks}

Let us more rigorously pin down the descending links. For a given $\sigma$, the descending link $\dlk(\sigma)$ of $\sigma$ in $\rips(X)'$ is spanned by faces $\sigma'<\sigma$ with $\diam(\sigma')<\diam(\sigma)$, called \emph{descending faces}, and proper cofaces $\widetilde{\sigma}>\sigma$ with $\diam(\widetilde{\sigma})=\diam(\sigma)$, called \emph{descending cofaces}. Viewing $\sigma$ as a finite subset of $X$, a descending face $\sigma'\subseteq \sigma$ is obtained by removing points so as to make the diameter go strictly down, and a descending coface $\widetilde{\sigma}\supsetneq\sigma$ is obtained by adding new points so as to avoid making the diameter go up.

Since every descending face of $\sigma$ is a face of every descending coface, $\dlk(\sigma)$ is the simplicial join of the subcomplex spanned by the descending faces, called the \emph{descending face link} $\dflk(\sigma)$, and the subcomplex spanned by the descending cofaces, called the \emph{descending coface link} $\dclk(\sigma)$. Now for example if either $\dflk(\sigma)$ or $\dclk(\sigma)$ is contractible, then $\dlk(\sigma)$ is contractible.

Assuming the Morse Criterion holds, for any $t<s$ we have that $\rips_s(X)$ is obtained from $\rips_t(X)$ by gluing in simplices (or more precisely the cones on such simplices) $\sigma$ with $t<\diam(\sigma)\le s$ along $\dlk(\sigma)$. This provides a potentially useful algorithm. For example, if every $\sigma$ with $t<\diam(\sigma)\le s$ has a contractible descending link, then $\rips_s(X)$ is homotopy equivalent to $\rips_t(X)$. This situation will be discussed more later.

\section{The Link Criterion}\label{sec:links}

As before, $(X,d)$ is a (non-empty) metric space. We introduce a condition that will ensure well behaved descending links. Here $B_t(x)\defeq \{x'\in X\mid d(x,x')<t\}$ and $\overline{B}_t(x)\defeq \{x'\in X\mid d(x,x')\le t\}$. (Note that $\overline{B}_t(x)$ does not necessarily equal the closure of $B_t(x)$, this is just notation.)

\begin{definition}[Link Criterion]
Let $X$ be a metric space and $F\subseteq X$ a finite subset with $\diam(F)=t$. We say $F$ satisfies the \emph{Link Criterion} if either there exists $z\in F$ such that $F\subseteq B_t(z)$ or there exists $z\in X\setminus F$ such that $\bigcap_{f\in F}\overline{B}_t(f) \subseteq \overline{B}_t(z)$.
\end{definition}

\begin{definition}[Strong Link Criterion]
We say $X$ satisfies the \emph{Strong Link Criterion at scale $t$} if for all $x,y\in X$ with $d(x,y)=t$, there exists $z\in X$ such that $\overline{B}_t(x)\cap \overline{B}_t(y) \subseteq B_t(z)$.
\end{definition}

Given a subset $I\subseteq \R$, we will say that $X$ \emph{satisfies the Link Criterion in the range $I$} if every finite $F\subseteq X$ with $\diam(F)\in I$ satisfies the Link Criterion. Similarly, we will say $X$ \emph{satisfies the Strong Link Criterion in the range $I$} if it satisfies the Strong Link Criterion at scale $t$ for all $t\in I$. In practice we will be most interested in the case when $I$ is an interval, e.g., of the form $[N,M]$, $(N,M]$, $[N,\infty)$, etc.

\begin{observation}\label{obs:link_condition}
If $X$ satisfies the Strong Link Criterion in the range $I$ then it satisfies the Link Criterion in the range $I$.
\end{observation}

\begin{proof}
Choose $x,y\in F$ with $d(x,y)=\diam(F)$, and set $t\defeq \diam(F)$. Now $F\subseteq \bigcap_{f\in F}\overline{B}_t(f)\subseteq \overline{B}_t(x)\cap \overline{B}_t(y) \subseteq B_t(z)\subseteq \overline{B}_t(z)$, regardless of whether $z$ came from $F$ or $X\setminus F$.
\end{proof}

The point of the Link Criterion is the following result.

\begin{lemma}\label{lem:cible_dlks}
Let $X$ be a metric space and $F\subseteq X$ a finite subset. Let $\sigma$ be the simplex in $\rips(X)$ with vertex set $F$. If $F$ satisfies the Link Criterion then the descending link of $\sigma$ is contractible.
\end{lemma}

\begin{proof}
Let $t=\diam(\sigma)$. By the Link Criterion we can choose $z\in X$ such that either $z\in X\setminus F$ and $\bigcap_{f\in F}\overline{B}_t(f) \subseteq \overline{B}_t(z)$ (so in particular $z\in \bigcap_{f\in F}\overline{B}_t(f)$), or else $z\in F$ and $F \subseteq B_t(z)$. In the $z\in F$ case, for any descending face $\sigma'$ of $\sigma$, $\sigma'\cup\{z\}$ is also a descending face of $\sigma$ since $\diam(\sigma')<t$ and hence also $\diam(\sigma'\cup\{z\})<t$. In this case $\dflk(\sigma)$ is a cone on the vertex $z$, hence contractible. In the $z\in X\setminus F$ case, $\sigma\cup\{z\}$ is a descending coface of $\sigma$ since $\diam(\sigma\cup\{z\})=\diam(\sigma)$, and moreover for any $w\in X$ such that $\sigma\cup\{w\}$ is a descending coface of $\sigma$ we have that $\sigma\cup\{w,z\}$ is as well, so $\dclk(\sigma)$ is contractible via the conical contraction $\widetilde{\sigma} \le \widetilde{\sigma}\cup\{z\} \ge \{z\}$.
\end{proof}

\begin{theorem}\label{thrm:contractible}
Let $X$ be a metric space satisfying the Morse Criterion, and the Link Criterion in the range $I$. Then for any $(t,s]\subseteq I$ the inclusion $\rips_t(X)\to\rips_s(X)$ is a homotopy equivalence, and for any $(t,\infty)\subseteq I$, $\rips_t(X)$ is contractible.
\end{theorem}

\begin{proof}
By Lemma~\ref{lem:morse_criterion} $(\diam,-\dim)$ is a Morse function on $\rips(X)'$. Lemma~\ref{lem:cible_dlks} says that for any $\sigma$ with $\diam(\sigma)\in I$, $\dlk(\sigma)$ is contractible. Hence Corollary~\ref{cor:morse} says that for any $(t,s]\subseteq I$ the inclusion $\rips_t(X)\to \rips_s(X)$ induces an isomorphism in $\pi_k$ for all $k$. This inclusion is therefore a homotopy equivalence by the Whitehead Theorem. The same argument applies if $(t,\infty)\subseteq I$, with the conclusion being that $\rips_t(X)\to \rips(X)$ is a homotopy equivalence, so $\rips_t(X)$ is contractible.
\end{proof}

\begin{remark}\label{rmk:noncanonical}
The means by which the descending links in Lemma~\ref{lem:cible_dlks} are contractible is reminiscent of Forman's discrete Morse theory from Example~\ref{ex:forman}, in that sometimes it is due to a conically contractible descending face link and sometimes due to a conically contractible descending coface link. However, it does not seem easy to extract a ``Forman-style'' discrete Morse function giving the same sublevel sets, since the choice of $z$ in the Link Criterion is not necessarily canonical. In particular, in Forman's framework the descending coface link of a redundant simplex is a cone on a single codimension-$1$ coface (the collapsible simplex with which it is paired), and here descending coface links can easily be contractible for other reasons. Similarly, in Forman's framework the descending face link of a collapsible simplex is a cone on the vertex opposite a single codimension-$1$ coface (the redundant simplex with which is it paired), and here descending face links can be contractible for more complicated reasons.
\end{remark}

In the coming sections we will make use of some still stronger versions of the Strong Link Criterion, which we record here.

\begin{definition}[$r$-Pinched Strong Link Criterion]
Let $r\ge 0$. We say a metric space $X$ satisfies the \emph{$r$-Pinched Strong Link Criterion} in the range $I$ if for any $x,y\in X$ with $d(x,y)=t\in I$, there exists $z\in X$ such that $\overline{B}_t(x)\cap \overline{B}_t(y) \subseteq B_{t-r}(z)$.
\end{definition}

\begin{definition}[Arbitrarily Pinched Strong Link Criterion]
A metric space $X$ satisfies the \emph{Arbitrarily Pinched Strong Link Criterion} if for all $r$ there exists $N$ such that $X$ satisfies the $r$-Pinched Strong Link Criterion in the range $[N,\infty)$.
\end{definition}

\section{Examples}\label{sec:link_examples}

In this section we discuss examples of metric spaces $X$ for which the Link Criterion and its stronger forms hold. Our primary example will be asymptotically $\CAT(0)$ geodesic spaces, defined in Definition~\ref{def:asymp_cat0}. Let us ease into things with the easy example of euclidean space.

\begin{example}\label{ex:euc}
Let $X=\R^n$ with the euclidean metric. We claim $X$ satisfies the $r$-Pinched Strong Link Criterion in the range $I=(\frac{2r}{2-\sqrt{3}},\infty)$, so in particular satisfies the Arbitrarily Pinched Strong Link Criterion. Indeed, if $d(x,y)=t\in I$ and $z$ is the midpoint of the line segment from $x$ to $y$, then for any $w\in \overline{B}_t(x)\cap \overline{B}_t(y)$ we have $d(z,w)\le \frac{\sqrt{3}}{2}t$. Hence $w\in B_{t-s}(z)$ for any $s<\frac{2-\sqrt{3}}{2}t$. Now since $t>\frac{2r}{2-\sqrt{3}}$ we have $r<\frac{2-\sqrt{3}}{2}t$, so $\overline{B}_t(x)\cap \overline{B}_t(y)\subseteq B_{t-r}(z)$ as desired.
\end{example}

A more general example is $\CAT(0)$ spaces, or even asymptotically $\CAT(0)$ spaces introduced in \cite{kar11}. A metric space $X$ is \emph{asymptotically CAT(0)} if all of its asymptotic cones are $\CAT(0)$. A useful working definition in the case of geodesic metric spaces, equivalent by Theorem~8 of \cite{kar11}, involves inspecting comparison triangles. Given a geodesic triangle $\Delta$ in $X$, say with vertices $x_1$, $x_2$, and $x_3$, a \emph{comparison triangle} $\overline{\Delta}$ is a geodesic triangle in $\R^2$ with vertices $\overline{x_1}$, $\overline{x_2}$, and $\overline{x_3}$ such that $d(x_i,x_j)=d(\overline{x_i},\overline{x_j})$ for all $i,j\in\{1,2,3\}$. Given a point $p\in\Delta$, say on the geodesic from $x_1$ to $x_2$, the \emph{comparison point} is the point $\overline{p}\in\overline{\Delta}$ on the geodesic from $\overline{x}_1$ to $\overline{x}_2$ satisfying $d(\overline{x_1},\overline{p})=d(x_1,p)$. Also set $d(\Delta)\defeq \max\{d(x_i,x_j)\mid i,j\in\{1,2,3\}\}$. Now our working definition (which is phrased slightly differently than Theorem~8 of \cite{kar11} but is equivalent) is as follows.

\begin{definition}[Asymptotically $\CAT(0)$]\label{def:asymp_cat0}
A geodesic metric space $X$ is \emph{asymptotically $\CAT(0)$} if there exists $f\colon \R\to\R$ with $\lim_{t\to\infty}\frac{f(t)}{t}=0$ such that for any $\Delta$ and $\overline{\Delta}$ as above and for any points $p,q\in\Delta$ we have $d(p,q)\le d(\overline{p},\overline{q})+f(d(\Delta))$.
\end{definition}

\begin{lemma}\label{lem:asymp_cat0}
If a geodesic metric space $X$ is asymptotically $\CAT(0)$ then it satisfies the Arbitrarily Pinched Strong Link Criterion.
\end{lemma}

\begin{proof}
For $r\ge 0$ choose $N$ large enough that $N>\frac{2r}{2-\sqrt{3}}$ and $\frac{f(t)}{t}<1-\frac{r}{t}-\frac{\sqrt{3}}{2}$ for all $t\ge N$. (The first assumption ensures that $0<1-\frac{r}{t}-\frac{\sqrt{3}}{2}$ for all $t\ge N$, making the second assumption possible.) Let $x,y\in X$ with $d(x,y)=t\ge N$. Choose a geodesic from $x$ to $y$ and let $z$ be its midpoint. Let $w\in \overline{B}_t(x)\cap \overline{B}_t(y)$. Let $\Delta$ be a geodesic triangle with vertices $x,y,w$ such that the geodesic whose midpoint is $z$ is the side of $\Delta$ from $x$ to $y$. Note that in the comparison triangle $\overline{\Delta}$ we have $d(\overline{z},\overline{w})\le t\frac{\sqrt{3}}{2}$, so we know $d(z,w)\le t\frac{\sqrt{3}}{2}+f(t)$. Our assumptions ensure that $f(t)<t(1-\frac{r}{t}-\frac{\sqrt{3}}{2})$, so $d(z,w)<t-r$ as desired.
\end{proof}

Since $\CAT(0)$ and hyperbolic geodesic spaces are asymptotically $\CAT(0)$, Lemma~\ref{lem:asymp_cat0} says they satisfy the Arbitrarily Pinched Strong Link Criterion.

\begin{remark}
The proof of Lemma~\ref{lem:asymp_cat0} would work with $\frac{\sqrt{3}}{2}$ replaced everywhere by any constant strictly less than $1$, so one could use a generalization of asymptotically $\CAT(0)$ that allows ``fatter'' comparison triangles (though not as fat as, e.g., the 90-90-90 triangles in $S^2$). We will revisit this in the proof of Proposition~\ref{prop:spheres}.
\end{remark}

Of course for $\mathfrak{X}$ a non-trivial geodesic metric space the Morse Criterion does not hold (see Remark~\ref{rmk:geodesic}), so this does not tell us anything directly about $\rips_t(\mathfrak{X})$ for such $\mathfrak{X}$. In practice though, it is often possible to pass from a space $\mathfrak{X}$ satisfying the $r$-Pinched Strong Link Criterion, in some range, to a subspace $X$, with the induced metric, satisfying both the Morse Criterion and the Strong Link Criterion in the same range. The key is to make sure that $X$ is \emph{$r$-dense} in $\mathfrak{X}$, meaning for all $x\in \mathfrak{X}$ there exists $x'\in X$ with $d(x,x')\le r$.

\begin{lemma}\label{lem:lattice}
Let $\mathfrak{X}$ be a metric space satisfying the $r$-Pinched Strong Link Criterion in the range $I$ for some $r$. Then any $r$-dense subspace $X$ of $\mathfrak{X}$ with the induced metric satisfies the Strong Link Criterion in the range $I$.
\end{lemma}

\begin{proof}
Let $x,y\in X$ with $d(x,y)=t\in I$. Choose $z\in \mathfrak{X}$ such that $\overline{B}_t(x)\cap \overline{B}_t(y) \subseteq B_{t-r}(z)$ (here balls are taken in $\mathfrak{X}$). Choose $z'\in X$ with $d(z,z')\le r$. Now for any $w\in X\cap \overline{B}_t(x)\cap \overline{B}_t(y)$ we have $d(w,z') < (t-r)+r=t$.
\end{proof}

\begin{proposition}\label{prop:lattice_rips}
Let $\mathfrak{X}$ be a metric space satisfying the $r$-Pinched Strong Link Criterion in the range $I$ for some $r$. Let $X$ be an $r$-dense subspace of $\mathfrak{X}$ with the induced metric that satisfies the Morse Criterion. Then for any $(t,s]\subseteq I$ the inclusion $\rips_t(X)\to\rips_s(X)$ is a homotopy equivalence, and for any $(t,\infty)\subseteq I$, $\rips_t(X)$ is contractible.
\end{proposition}

\begin{proof}
We are assuming $X$ satisfies the Morse Criterion, and Lemma~\ref{lem:lattice} says $X$ satisfies the Strong Link Criterion in the range $I$, so the result follows from Theorem~\ref{thrm:contractible}.
\end{proof}

\begin{example}\label{ex:lattice}
Let $X$ be a subspace of $\R^n$ with the induced euclidean metric. Assume $X$ is proper and the set of diameters of finite subsets of $X$ is closed and discrete in $\R$, so Observation~\ref{obs:discrete} says $X$ satisfies the Morse Criterion, and further assume that $X$ is $r$-dense in $\R^n$ for some $r$. (A quintessential example is $X=\Z^n$.) From Example~\ref{ex:euc}, $\mathfrak{X}=\R^n$ with the euclidean metric satisfies the $r$-Pinched Strong Link Criterion in the range $(\frac{2r}{2-\sqrt{3}},\infty)$, so Proposition~\ref{prop:lattice_rips} says the Vietoris--Rips complex $\rips_t(X)$ is contractible for all $t>\frac{2r}{2-\sqrt{3}}$. For example $\rips_t(\Z^n)$ is contractible for all $t>\frac{\sqrt{n}}{2-\sqrt{3}}$.
\end{example}

This example is interesting since it is generally difficult to compute homotopy types of Vietoris--Rips complexes of subsets of $\R^n$, even of $\R^2$. For more in this vein see \cite[Problem~7.3]{adamaszek17}, \cite{adamaszek17circle}, and \cite{chambers10}. More generally Example~\ref{ex:lattice} works for any analogous subspace of an asymptotically $\CAT(0)$ space with the induced metric, for instance lattices in non-positively curved spaces.

It is also possible to leverage Proposition~\ref{prop:lattice_rips} to say something about the $\rips_t(\mathfrak{X})$:

\begin{corollary}\label{cor:approx}
Let $\mathfrak{X}$ be a metric space satisfying the $r$-Pinched Strong Link Criterion in the range $I$ for some $r$. Suppose there exists a subspace $X$ of $\mathfrak{X}$ that is $r$-dense and satisfies the Morse Criterion. Then for any $(t,s]\subseteq I$, the inclusion $\rips_t(\mathfrak{X})\to\rips_s(\mathfrak{X})$ is a homotopy equivalence, and for any $(t,\infty)\subseteq I$, $\rips_t(\mathfrak{X})$ is contractible.
\end{corollary}

\begin{proof}
First note that for any finite $F\subseteq \mathfrak{X}$, the union $X\cup F$ is still $r$-dense and still satisfies the Morse Criterion. By Proposition~\ref{prop:lattice_rips}, $\rips_t(X\cup F)\to\rips_s(X\cup F)$ is a homotopy equivalence for all finite $F$ and all $t<s$ with $(t,s]\subseteq I$. Now we proceed similarly to the proof of Proposition~7.1 in \cite{adamaszek17circle}: clearly $\rips_t(\mathfrak{X})$ is the colimit of the $\rips_t(X\cup F)$, with $F$ ranging over all finite subsets of $X$, and the level-wise maps $\rips_t(X\cup F)\to\rips_s(X\cup F)$ are homotopy equivalences, so the induced map $\rips_t(\mathfrak{X})\to\rips_s(\mathfrak{X})$ is a homotopy equivalence. (The key reason this works, as explained in \cite{adamaszek17circle}, is that the maps $\rips_t(X\cup F)\to\rips_t(X\cup F')$ for $F\subseteq F'$ are inclusions of closed subcomplexes, hence are cofibrations.) The $(t,\infty)\subseteq I$ case follows similarly, using the fact that $\rips(\mathfrak{X})$ is contractible.
\end{proof}

\begin{example}[Vietoris--Rips complexes of $\CAT(0)$ spaces]
Corollary~\ref{cor:approx} implies that $\rips_t(\R^n)$ is contractible for all $t>0$, since for any $r>0$, $\R^n$ admits $r$-dense subspaces satisfying the Morse Criterion (just take appropriately scaled copies of $\Z^n$). More generally, thanks to Lemma~\ref{lem:asymp_cat0}, it is not hard to show that any (proper) asymptotically $\CAT(0)$ space $\mathfrak{X}$ admitting a properly discontinuous cocompact action by a group $G$ of isometries will have contractible $\rips_t(\mathfrak{X})$ for $t$ large enough: just take a (compact) closed ball in $\mathfrak{X}$ whose $G$-translates cover $\mathfrak{X}$, choose a finite $r$-dense subset of the ball for arbitrarily small $r$, let $X$ be the set of translates of this finite set, and apply Corollary~\ref{cor:approx}.
\end{example}

Finally we discuss an illustrative example where the Link Criterion does not hold in any range of the form $[N,\infty)$.

\begin{example}
For $C_n$ the $n$-cycle graph set $\Gamma=C_4\vee C_6\vee C_8\vee\cdots$, and let $X=\Gamma^{(0)}$ with the induced path metric. Let $N\ge 2$ and let $F$ be the set of vertices of $C_{2N}$, so $\diam(F)=N$. Then $\bigcap_{f\in F}\overline{B}_N(f)=F$, so for $z\in X$ we have $\bigcap_{f\in F}\overline{B}_N(f)\subseteq \overline{B}_N(z)$ if and only if $z\in F$. This rules out the $z\in X\setminus F$ case of the Link Criterion, and the $z\in F$ case also does not hold since every $z\in F$ is distance $N$ from its antipode in $C_{2N}$.
\end{example}

This example reflects the intuition that one impediment to the Link Criterion holding in ranges of the form $[N,\infty)$ is if $X$ has ``arbitrarily large empty spheres.''

\section{Application: Topological data analysis}\label{sec:tda}

Finite metric spaces, typically living in $\R^n$, are one of the main topics of interest in topological data analysis. One can view the finite set as a point cloud, or set of data points, and try to use topological methods to understand the distribution of data or recover some object from which the data was sampled. A common tool in this field is Forman's discrete Morse theory (Example~\ref{ex:forman}), see for instance \cite{mischaikow13,harker14}. When using discrete Morse theory to analyze Vietoris--Rips complexes of a point cloud $X$, one would hope to find a Morse function with respect to which ``most'' descending links are contractible. This reduces noise and ensures that important topological features of $\rips_t(X)$ persist for large ranges of $t$ values. The most obvious choice of Morse function on $\rips(X)'$ is $\dim$, but this is as bad as possible, since none of the descending links are contractible.

A more clever choice of discrete Morse function on $\rips(X)'$ is $(\diam,-\dim)$, and it really is a Morse function, since any finite metric space satisfies the Morse Criterion. Now if there is a large range of $t$ values such that every simplex of $\rips(X)$ with $\diam$ value in that range has contractible descending link (e.g., by virtue of its vertex set satisfying the Link Criterion), then all the $\rips_t(X)$ for $t$ in that range are homotopy equivalent, indicating some persistent behavior. Intuitively a simplex $\sigma$ with vertex set $F$ will satisfy the Link Criterion, and hence have contractible descending link, if either some point of $F$ is strictly closer than $\diam(F)$ to every point in $F$, or if some point of $X$ outside $F$ can be added to $F$ without changing the set of points of $X$ within $\diam(\sigma)$ of each point in $F$. Such $\sigma$ are, roughly, those not ``surrounding an empty sphere.'' See Figure~\ref{fig:point_cloud} for an example.

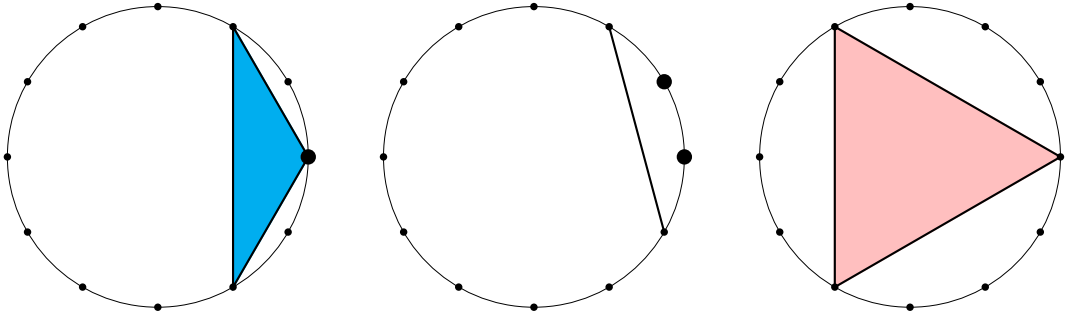
\begin{figure}[htb]
 \begin{tikzpicture}[line width=0.8pt]\centering
  \draw[line width=0.2pt] (0,0) circle (2cm);
	\filldraw[gray] (60:2) -- (0:2) -- (300:2) -- (60:2);
	\draw (60:2) -- (0:2) -- (300:2) -- (60:2);
  \filldraw (0:2) circle (2.5pt);
  \filldraw (30:2) circle (1pt);
  \filldraw (60:2) circle (1pt);
  \filldraw (90:2) circle (1pt);
  \filldraw (120:2) circle (1pt);
  \filldraw (150:2) circle (1pt);
  \filldraw (180:2) circle (1pt);
  \filldraw (210:2) circle (1pt);
  \filldraw (240:2) circle (1pt);
  \filldraw (270:2) circle (1pt);
  \filldraw (300:2) circle (1pt);
  \filldraw (330:2) circle (1pt);
	
	\begin{scope}[xshift=5cm]	
  \draw[line width=0.2pt] (0,0) circle (2cm);
	\draw (60:2) -- (330:2);
  \filldraw (0:2) circle (2.5pt);
  \filldraw (30:2) circle (2.5pt);
  \filldraw (60:2) circle (1pt);
  \filldraw (90:2) circle (1pt);
  \filldraw (120:2) circle (1pt);
  \filldraw (150:2) circle (1pt);
  \filldraw (180:2) circle (1pt);
  \filldraw (210:2) circle (1pt);
  \filldraw (240:2) circle (1pt);
  \filldraw (270:2) circle (1pt);
  \filldraw (300:2) circle (1pt);
  \filldraw (330:2) circle (1pt);
	\end{scope}
	
	\begin{scope}[xshift=10cm]	
  \draw[line width=0.2pt] (0,0) circle (2cm);
	\filldraw[lightgray] (0:2) -- (120:2) -- (240:2) -- (0:2);
	\draw (0:2) -- (120:2) -- (240:2) -- (0:2);
  \filldraw (0:2) circle (1pt);
  \filldraw (30:2) circle (1pt);
  \filldraw (60:2) circle (1pt);
  \filldraw (90:2) circle (1pt);
  \filldraw (120:2) circle (1pt);
  \filldraw (150:2) circle (1pt);
  \filldraw (180:2) circle (1pt);
  \filldraw (210:2) circle (1pt);
  \filldraw (240:2) circle (1pt);
  \filldraw (270:2) circle (1pt);
  \filldraw (300:2) circle (1pt);
  \filldraw (330:2) circle (1pt);
	\end{scope}
 \end{tikzpicture}
 \caption{Some simplices in $\rips(X)$, for $X$ the indicated twelve points on the circle, with the induced arclength metric. The first simplex (dark gray) has a vertex marked in bold closer to its other vertices than the diameter of the simplex, hence satisfies the Link Criterion. The second simplex (edge) satisfies the Link Criterion for the other reason, using either of the vertices in bold. The third simplex (light gray) does not satisfy the Link Criterion.}\label{fig:point_cloud}
\end{figure}

As explained in Remark~\ref{rmk:forman_to_bb}, any application of Forman's discrete Morse theory to topological data analysis can be recast in terms of Bestvina--Brady discrete Morse theory. We suspect that the Bestvina--Brady approach could lead to more efficient algorithms, since as explained in Remark~\ref{rmk:noncanonical} it is more general: descending links could turn out to be contractible via more complicated means than just having a single descending coface (for redundant simplices) or a single non-descending face (for collapsible simplices). Also, using Bestvina--Brady-style discrete Morse functions other than ones of the form $(h,-\dim)$ (i.e., ones coming from a Forman-style discrete Morse function $h$) for applications in topological data analysis seems to be a completely untapped avenue.

\medskip

For the rest of this section we focus on a problem related to topological data analysis, involving computing homotopy types of Vietoris--Rips complexes of spheres $S^n$. In \cite{adamaszek17circle}, Adamaszek and Adams compute the homotopy type of $\rips_t(S^1)$ for all $t$. To illustrate an application of the Morse function $(\diam,-\dim)$, in this next example we use it to recover the fact that all the $\rips_t(S^1)$ for $0<t<1/3$ are homotopy equivalent to each other (in fact to $S^1$).

\begin{example}\label{ex:circle}
Consider the arc length metric on $S^1$, normalized so the circumference is $1$. Note that for any $x,y\in S^1$ with $d(x,y)=t<1/3$, the intersection $\overline{B}_t(x)\cap\overline{B}_t(y)$ equals the geodesic from $x$ to $y$. In particular if $z$ is the midpoint of this geodesic then this intersection lies in $B_{t-r}(z)$ for any $0<r<t/2$. Hence for any $0<r<t/2$, $S^1$ satisfies the $r$-Pinched Strong Link Criterion in the range $(2r,1/3)$. Since $S^1$ certainly contains $r$-dense finite subsets for arbitrary $r$ (and since finite sets satisfy the Morse Criterion), Corollary~\ref{cor:approx} says that for any $t<s$ with $2r<t<s<1/3$, the inclusion $\rips_t(S^1)\to\rips_s(S^1)$ is a homotopy equivalence. Varying $r$ tells us this is a homotopy equivalence for all $0<t<s<1/3$. Latschev's Theorem \cite{latschev01} says $\rips_t(S^1)\simeq S^1$ for $0<t$ small enough, so we conclude $\rips_t(S^1)\simeq S^1$ for all $0<t<1/3$.
\end{example}

As a remark, for $t\ge 1/3$ it seems the Morse function $(\diam,-\dim)$ cannot be used to fully recover the computation of $\rips_t(S^1)$. In particular, the order dictated by $(\diam,-\dim)$ in which the simplices of diameter $1/3$ are glued in is ``wrong'', and the homotopy type of $\rips_{1/3}(S^1)$ cannot be obtained this way without ``extra knowledge'' about how the simplices fit together beyond the local information coming from the descending links. It seems likely though that a more intricate Morse function could be used to fully recover all the $\rips_t(S^1)$, but we will leave this for the future.

For $S^n$ with $n\ge 2$ the homotopy types of the $\rips_t(S^n)$ for arbitrary $t$ are currently unknown (here we use the arclength metric on $S^n$, scaled so that antipodal points are at distance $1/2$). In \cite{hausmann95} Hausmann proves that $\rips_t(S^n)\simeq S^n$ for $0<t<1/4$ (note that in \cite{hausmann95} the definition of Vietoris--Rips complex uses ``$<$'' instead of ``$\le$'' so in his language the range is $0<t\le 1/4$). The proof in \cite{hausmann95} is extremely technical and involves choosing a total ordering of the points of the sphere. We can recover this result now using a much simpler and ``local'' Morse theoretic argument:

\begin{proposition}\label{prop:spheres}
For any $0<t<1/4$ we have $\rips_t(S^n)\simeq S^n$.
\end{proposition}

\begin{proof}
Since $t<1/4$, there exists $0<q<1$ such that given any geodesic triangle in $S^n$ with diameter $s<t$, the length of any median of this triangle is less than $qs$ (this follows from the spherical law of cosines). For any $r\ge 0$, if $r/(1-q) < s<t$ then in any geodesic triangle with diameter $s$ the length of any median is less than $s-r$, since $r/(1-q) < s$ implies $qs<s-r$. Now let $x,y\in S^n$ with $d(x,y)=s\in (r/(1-q),t]$, and let $z$ be the midpoint of the geodesic from $x$ to $y$. By the above, for any $w\in S^n$ with $d(w,x),d(w,y)\le s$ we have $d(w,z)< s-r$. In particular $S^n$ satisfies the $r$-Pinched Strong Link Criterion in the range $(r/(1-q),t]$. Since $S^n$ contains $r$-dense finite subsets for arbitrary $r$ (and since finite sets satisfy the Morse Criterion), Corollary~\ref{cor:approx} says that for any $s<t$ with $r/(1-q)<s<t<1/4$, the inclusion $\rips_s(S^n)\to\rips_t(S^n)$ is a homotopy equivalence. Varying $r$ tells us this is a homotopy equivalence for all $0<s<t<1/4$. By Latschev's Theorem \cite{latschev01}, $\rips_t(S^n)\simeq S^n$ for $0<t$ small enough, so we conclude $\rips_t(S^n)\simeq S^n$ for all $0<t<1/4$.
\end{proof}

As a remark, we do not think the $1/4$ is optimal; the persistent homotopy type should continue past $t=1/4$ up to $t=r_n$, where $r_n$ is the diameter of an inscribed regular $(n+1)$-simplex in $S^n$. However, the Link Criterion fails for $t=1/4$: in $S^2$ if $x$ and $y$ lie on the equator at distance $1/4$, then the north and south pole are both in $\overline{B}_{1/4}(x)\cap\overline{B}_{1/4}(y)$, but the poles cannot lie in $B_{1/4}(z)$ for any $z$ since they are at distance $1/2$ from each other. It would be interesting to see whether the Morse function $(\diam,-\dim)$ or a variation could be useful for computing $\rips_t(S^n)$ for $1/4 \le t<1/2$. We should also point out that Adamaszek--Adams--Frick obtained the result up to $r_n$ for the so called metric thickening \cite{adamaszek18}, and probably their methods could have also been used to obtain this for the usual Vietoris--Rips complexes.

\begin{remark}
For future applications of these sorts of discrete Morse theoretic tools to topological data analysis, we reiterate that $(\diam,-\dim)$ is just one choice of Morse function, and other Morse functions could be better for different applications. For a given problem, if $(\diam,-\dim)$ does not yield nice enough descending links, perhaps adding some clever third term $h$ does, e.g., $(h,\diam,-\dim)$, $(h,\diam,\dim)$, $(\diam,h,-\dim)$, or $(\diam,h,\dim)$. It is also possible that in some situations $(\diam,-\dim)$ does in fact work, but the descending links are contractible for more complicated reasons than the Link Criterion. In any case it would be interesting to find a Morse theoretic proof of the full computation of the $\rips_t(S^1)$ from \cite{adamaszek17circle}, since this could be a step toward computing the $\rips_t(S^n)$.
\end{remark}

\section{Application: Geometric group theory}\label{sec:groups}

Geometric group theory is often defined to be the study of groups as metric spaces. In this section we use our results to deduce various properties of groups, in particular related to topological finiteness properties. Recall that a cellular action of a group on a CW complex $X$ is \emph{geometric} if it is proper (meaning cell stabilizers are finite) and cocompact (meaning the orbit space is compact), and an action of a group by isometries on a metric space $X$ is \emph{geometric} if it is properly discontinuous (meaning only finitely many elements move each compact subset to a non-disjoint subset) and cocompact.

Recall from the introduction that a group is of type $\F_n$ if it admits a geometric action on an $(n-1)$-connected CW complex. Type $\F_\infty$ means type $\F_n$ for all $n$. A group is of type $\F_*$ if it admits a geometric action on a contractible CW complex. For torsion-free groups, type $\F_*$ is equivalent to type $\F$, meaning admitting a compact classifying space.

\begin{proposition}\label{prop:geom_action}
Let $\mathfrak{X}$ be a proper metric space satisfying the Arbitrarily Pinched Strong Link Criterion. Let $G$ be a group acting geometrically on $\mathfrak{X}$. Then $G$ is of type $\F_*$.
\end{proposition}

\begin{proof}
Let $x_0\in \mathfrak{X}$ and let $X\defeq G.x_0$ be the orbit, viewed as a metric space with the induced metric. Since the action is cocompact, $X$ is $r$-dense in $\mathfrak{X}$ for some $r$. Since the action is properly discontinuous only finitely many elements of $X$ lie in any metric ball in $\mathfrak{X}$, so $X$ is proper. Since $G$ is transitive on $X$, finiteness of balls implies that the set of diameters of finite subsets of $X$ is closed and discrete in $\R$. Hence Observation~\ref{obs:discrete} says that $X$ satisfies the Morse Criterion. By Lemma~\ref{lem:lattice} and Theorem~\ref{thrm:contractible}, some proper Vietoris--Rips complex $\rips_t(X)$ is contractible. The action of $G$ on $\mathfrak{X}$ induces an action on $\rips_t(X)$. Since only finitely many elements of $X$ lie in any metric ball, $\rips_t(X)$ is locally finite, so the fact that $G$ acts vertex transitively on it implies the action is cocompact. Also, proper discontinuity ensures that all stabilizers in $G$ of simplices in $\rips_t(X)$ are finite. We conclude that $G$ acts geometrically on the contractible simplicial complex $\rips_t(X)$, so $G$ is of type $\F_*$.
\end{proof}

One application of Proposition~\ref{prop:geom_action} is to asymptotically $\CAT(0)$ groups. A group is \emph{asymptotically CAT(0)} if it admits a geometric action on a (proper) asymptotically $\CAT(0)$ space.

\begin{theorem}\label{thrm:asymp_cat0_group}
Let $G$ be an asymptotically $\CAT(0)$ group. Then $G$ admits a geometric action on a contractible simplicial complex, and hence is of type $\F_*$.
\end{theorem}

\begin{proof}
First let $\mathfrak{X}$ be an asymptotically $\CAT(0)$ space on which $G$ acts geometrically (so $\mathfrak{X}$ is proper, see \cite[Exercise~I.8.4(1)]{bridson99}). By Lemma~\ref{lem:asymp_cat0}, $\mathfrak{X}$ satisfies the Arbitrarily Pinched Strong Link Criterion. Hence Proposition~\ref{prop:geom_action} applies, and by the proof of the proposition there exist $X$ and $t$ such that $G$ acts geometrically on the contractible simplicial complex $\rips_t(X)$.
\end{proof}

This recovers various known results, e.g., that asymptotically $\CAT(0)$ groups are of type $\F_\infty$ \cite{kar11}, and that $\CAT(0)$ groups and hyperbolic groups are of type $\F_*$ \cite{ontaneda05,bridson99}. The fact that all asymptotically $\CAT(0)$ groups are of type $\F_*$ seems to be new (the $\F_\infty$ proof in \cite{kar11} does not lead to $\F_*$, as far as we can tell).

\subsection{Word metrics on groups}\label{sec:word}

A standard example of a finitely generated group acting geometrically on a metric space is given by taking a Cayley graph of the group with respect to a finite generating set. The \emph{word metric} on such a group is given by endowing the Cayley graph with the path metric (with each edge having length $1$) and then taking the induced metric on the vertex set. A group acting geometrically on a space satisfying the Strong Link Criterion in some range (even the Arbitrarily Pinched Strong Link Criterion) might not itself satisfy the Strong Link Criterion using the word metric. For example one can check that $\Z^2$ with the standard word metric does not satisfy the Strong Link Criterion in any relevant range. However, if a group with a word metric does satisfy the Link Criterion in some range $[N,\infty)$ then there is a lot we can say, as we discuss now.

First we remark that an adequate understanding of the Vietoris--Rips complexes of $G$ with the word metric can reveal topological finiteness properties of $G$. By Brown's Criterion \cite{brown87}, if a group acts properly on an $(n-1)$-connected CW complex $X$ with an invariant cocompact filtration $(X_t)_{t\in\R}$, then the group is of type $\F_n$ if and only if this filtration is \emph{essentially $(n-1)$-connected}, meaning for all $t$ there exists $s\ge t$ such that the inclusion $X_t\to X_s$ induces the trivial map in $\pi_k$ for $k\le n-1$. The following result gives a nice geometric definition of type $\F_n$.

\begin{lemma}\label{lem:rips_fin_props}
Let $G$ be a group with the word metric corresponding to some finite generating set. Then $G$ is of type $\F_n$ if and only if the filtration $\{\rips_t(G)\}_{t\in\R}$ is essentially $(n-1)$-connected. If some $\rips_t(G)$ is contractible then $G$ is of type $\F_*$.
\end{lemma}

\begin{proof}
The complex $\rips(G)$ is contractible and the action of $G$ on $\rips(G)$ is proper. The sublevel sets $\rips_t(G)$ are $G$-invariant, and are cocompact since they are locally compact and $G$ is vertex-transitive. The first result now follows from Brown's Criterion. The second result follows since the action of $G$ on $\rips_t(G)$ is geometric.
\end{proof}

Since all diameters of finite subsets of $G$ are in $\N_0$ and balls in $G$ are finite the Morse Criterion holds, so $(\diam,-\dim)$ is always a descending-type Morse function on any $\rips(G)$ as above. We now focus on the situation dictated by the Link Criterion. Since it will come up a lot from now on, let us define:

\begin{definition}[Asymptotic (Strong) Link Criterion]
If a metric space satisfies the (Strong) Link Criterion in the range $[N,\infty)$, say that it satisfies the \emph{Asymptotic (Strong) Link Criterion past $N$}. If it satisfies the Asymptotic (Strong) Link Criterion past $N$ for some $N$ we will just say it satisfies the Asymptotic (Strong) Link Criterion.
\end{definition}

\begin{theorem}\label{thrm:F_infty}
Let $G$ be a group with a word metric corresponding to some finite generating set. If $G$ satisfies the Asymptotic Link Criterion past $N$ then $\rips_t(G)$ is contractible for all $t\ge N$, and hence $G$ is of type $\F_*$.
\end{theorem}

\begin{proof}
Our assumptions ensure that the hypotheses of Theorem~\ref{thrm:contractible} are met, so $\rips_t(G)$ is contractible for all $t\ge N$. Hence by Lemma~\ref{lem:rips_fin_props}, $G$ is of type $\F_*$.
\end{proof}

In particular if $G$ is torsion-free and satisfies the Asymptotic Link Criterion then $G$ is of type $\F$. As a remark, we do not know whether every group of type $\F_*$ admits a contractible proper Vietoris--Rips complex. In fact we are not even sure whether $\Z^n$ with the standard word metric admits a contractible proper Vietoris--Rips complexes for all $n$. (For $n=2$ we do have an \emph{ad hoc} proof that $\rips_2(\Z^2)$ is contractible.)

\begin{remark}
If $S$ is a finite generating set of $G$ such that the word metric corresponding to $S$ produces a contractible proper Vietoris--Rips complex $\rips_t(G)$, then after replacing $S$ by the set $S'$ of all products of at most $t$ elements of $S$ we get a word metric corresponding to $S'$ with respect to which $\rips_1(G)$ is contractible. Note that $\rips_1(G)$ is the ``flagification'' of the Cayley graph of $G$ with respect to $S'$. Hence Theorem~\ref{thrm:F_infty} says that if a group admits a word metric satisfying the Asymptotic Link Criterion then $G$ admits a Cayley graph with contractible flagification.
\end{remark}

Hyperbolic groups famously have contractible $\rips_t(G)$ for large enough $t$. This is originally due to Rips, see~\cite[Proposition~III.$\Gamma$.3.23]{bridson99}. It was generalized by Alonso \cite{alonso92} who showed this for any group admitting a so called contracting combing. Following \cite{alonso92}, a \emph{path} in $G$ is an eventually constant map $p\colon \N_0\to G$ such that $d(p(n),p(n+1))\le 1$ for all $n\in \N_0$. Let $P(G)$ be the set of all paths $p$ with $p(0)=1$, and define $\pi\colon P(G)\to G$ by sending $p$ to $p(n)$ for $n$ large enough that $p(n)=p(n+1)=\cdots$. Now a \emph{combing} of $G$ is a section $s\colon G\to P(G)$ of $\pi$. A combing $s$ is \emph{geodesic} if every $s(g)$ is a geodesic. A combing $s$ is \emph{bounded} if there exists a monotone function $\phi\colon \N_0\to\N$ with $\phi(n)\ge n$ for all $n$ such that for all $g,h\in G$ and all $n\in\N_0$ we have $d(s(g)(n),s(h)(n))\le \phi(d(g,h))$. A combing $s$ is \emph{contracting} if there exists $C\ge 2$ such that for all $g,h\in G$ and $n,n'\in\N_0$, with $n'\le n$ and $\lfloor C/2\rfloor \le n$, if $d(s(g)(n),s(h)(n'))\le C$ then also $d(s(g)(n-\lfloor C/2\rfloor),s(h)(n'))\le C$. Alonso proves that every contracting combing is bounded \cite[Lemma~2]{alonso92}.

\begin{cit}\cite[Theorem~1]{alonso92}
If $G$ admits a bounded combing then for all $t$ there exists $s\ge t$ such that the inclusion $\rips_t(G)\to\rips_s(G)$ induces the trivial map in $\pi_k$ for all $k$. If $G$ admits a contracting combing then $\rips_t(G)$ is contractible for some $t$.
\end{cit}

In particular groups with a bounded combing are of type $\F_\infty$, and groups with a contracting combing are of type $\F_*$. Since hyperbolic groups admit contracting combings, this recovers Rips' result. Also, groups admitting bounded combings include the important family of automatic groups (in this case the function $\phi$ can be taken to be $\phi(n)=Cn+D$ for constants $C$ and $D$), so this implies automatic groups are of type $\F_\infty$.

The following describes the most general type of combing that ensures the Asymptotic Strong Link Criterion holds. (For lack of a better name we will just call it a ``good'' combing.)

\begin{proposition}[Good combing]\label{prop:good_combing}
Let $G$ be a finitely generated group. Suppose there exists a geodesic combing $s$ and a number $N\ge 0$ such that whenever $g\in G$ with $d(1,g)=t\ge N$, there exists $n$ such that for all $h\in G$ with $d(1,h),d(h,g)\le t$ we have $d(h,s(g)(n))<t$. Then $G$ satisfies the Strong Link Criterion in the range $[N,\infty)$, so $\rips_t(G)$ is contractible for all $t\ge N$ and $G$ is of type $\F_*$.
\end{proposition}

\begin{proof}
Let $x,y\in G$ with $d(x,y)=t\ge N$. Since the action of $G$ on itself preserves the word metric, without loss of generality $x=1$, and let us write $y=g$. Choose $n$ as in the hypothesis, and set $z=s(g)(n)$. Now for all $h\in G$ with $d(1,h),d(h,g)\le t$ we have $d(h,z)<t$, so $\overline{B}_t(1)\cap \overline{B}_t(g)\subseteq B_t(z)$ as desired. This shows that $G$ satisfies the Strong Link Criterion in the range $[N,\infty)$, and the other results follow from Theorem~\ref{thrm:F_infty}.
\end{proof}

It is easy to see that hyperbolic groups admit good combings, so this recovers Rips' result (and can be viewed as giving a discrete Morse theoretic proof of Rips' result). Indeed, suppose the Cayley graph is $\delta$-hyperbolic and consider a geodesic triangle with diameter $t$, say with vertices $1,g,h$ with $d(1,g)=t$. Let $z$ be the midpoint of the geodesic side from $1$ to $g$. Choose $v$ on one of the other sides, without loss of generality the side from $1$ to $h$, such that $d(v,z)\le \delta$, and note that $d(1,v)\ge d(1,z)-\delta=t/2 - \delta$ so $d(h,v)\le t/2 + \delta$. Now we have $d(h,z)\le d(h,v)+d(v,z)\le t/2+2\delta$. Thus, taking $n$ such that $s(g)(n)$ is within $1/2$ of $z$ (note any group is $1/2$-dense in its Cayley graph) we have $d(h,s(g)(n))\le t/2+2\delta+1/2$, which we can ensure is less than $t$ by taking $N> 4\delta+1$.

\bibliographystyle{alpha}

\end{document}